\documentclass[11pt,reqno]{amsart}

\usepackage{amsmath}
\usepackage{amsthm}
\usepackage{amsfonts}
\usepackage{amssymb}
\usepackage{color}
\usepackage{graphicx}
\usepackage[hidelinks]{hyperref}
\usepackage[margin=1.0in]{geometry}
\usepackage{todonotes}
\usepackage{verbatim}
\usepackage{enumitem}
\usepackage{comment}
\numberwithin{equation}{section}

\definecolor{skyblue}{rgb}{0.85,0.85,1}

\newtheorem{lemma}{Lemma}
\newtheorem{prop}{Proposition}
\newtheorem{theorem}{Theorem}
\newtheorem{cor}{Corollary}

\newtheorem{rem}{Remark}

\newtheorem{define}{Definition}
\numberwithin{lemma}{section}
\numberwithin{prop}{section}
\numberwithin{theorem}{section}
\numberwithin{cor}{section}
\numberwithin{conj}{section}
\numberwithin{rem}{section}

\DeclareMathOperator{\Mor}{Mor}

\newcommand{\bbC}{\mathbb{C}}

\newcommand{\bbR}{\mathbb{R}}

\newcommand{\bbN}{\mathbb{N}}
\newcommand{\vp}{\varphi}

\newcommand{\p}{\partial}

\DeclareMathOperator{\Gr}{Gr}
\DeclareMathOperator{\GL}{GL}

\newcommand{\eps}{\epsilon}
\newcommand{\Maslov}{\mathrm{Maslov}(\vp)}

\begin{document}

\title{Opening the Maslov Box for Traveling Waves in Skew-Gradient Systems}
\author[P. Cornwell]{Paul Cornwell}
\email{pcorn@live.unc.edu}
\address{Department of Mathematics, UNC Chapel Hill, Phillips Hall CB \#3250, Chapel Hill, NC 27516}

\begin{abstract}
We obtain geometric insight into the stability of traveling pulses for reaction-diffusion equations with skew-gradient structure. For such systems, a Maslov index of the traveling wave can be defined and related to the eigenvalue equation for the linearization $L$ about the wave. We prove two main results about this index. First, for general skew-gradient systems, it is shown that the Maslov index gives a lower bound on the number of real, unstable eigenvalues of $L$. Second, we show how the Maslov index gives an exact count of all unstable eigenvalues for fast traveling waves in a FitzHugh-Nagumo system. The latter proof involves the Evans function and reveals a new geometric way of understanding algebraic multiplicity of eigenvalues.
\end{abstract}

\maketitle

\tableofcontents

\section{Introduction}
The paragon of stability analysis for nonlinear waves is a result that relates spectral information to qualities of the wave itself. In principle, this could explain why some patterns and structures are prevalent in nature, while others are not. The classic example of this is Sturm-Liouville theory, which equates the number of unstable modes of a steady state solution of a scalar reaction-diffusion equation to the number of critical points it has. (See \S 2.3.2 of \cite{KP13}). For systems of equations, generalizations of Sturm-Liouville theory lead naturally to the Maslov index, which is a winding number for curves of Lagrangian subspaces. One drawback of the Maslov index as a stability index is that it has typically been applied only in a relatively small class of systems, namely those for which the steady state equation has a Hamiltonian structure. In this work, we show how the Maslov index can give a lower bound on the number of unstable eigenvalues for the linearization about a traveling wave in reaction-diffusion equations with skew-gradient structure. Such systems are necessarily not Hamiltonian. Additionally, we show how the same index gives an exact count of the unstable eigenvalues in a FitzHugh-Nagumo system. The proofs use an adaptation of the ``Maslov box" (see, for example, \cite{JLM13,HLS16,BCJLMS17,JLS17}) and an entirely intersection-based formulation of the Maslov index.

The systems of interest are reaction-diffusion equations of the form \begin{equation}\label{gen pde}
u_t=u_{xx}+Q Sf(u),
\end{equation} where $x,t\in\bbR$ are space and time respectively, $u\in\bbR^n$, and $f(u)=\nabla F(u)$ is the gradient of a function $F:\bbR^n\rightarrow\bbR$. The matrix $S\in\GL_n(\bbR)$ is positive and diagonal, and $Q\in\GL_{n}(\bbR)$ has the form \begin{equation}
Q=\mathrm{diag}\{d_1,\dots,d_n\},
\end{equation} where $d_i=\pm 1$ for all $i$. Such systems were dubbed ``skew-gradient" by Yanagida \cite{Yan02a,Yan02b}. We assume that (\ref{gen pde}) possesses a traveling pulse solution $\hat{u}$ which depends on one variable $z=x-ct$. Such solutions have a fixed profile and move with a constant speed $c$. Without loss of generality, we assume that \begin{equation}\label{speed negative}
c<0,
\end{equation} meaning that the wave moves to the left. There has been considerable progress in the stability analysis of \emph{standing waves} of (\ref{gen pde})--see below for more detail--but the known results do not apply to traveling waves. We aim to use the Maslov index to give a systematic treatment of traveling waves in such systems.

Written in a moving frame, a traveling pulse of (\ref{gen pde}) is a steady state of the equation \begin{equation}\label{tw pde}
u_t=u_{zz}+cu_z+QSf(u),
\end{equation} which decays exponentially to a constant state $u_\infty\in\bbR^n$ as $z\rightarrow\pm\infty$. For simplicity, we take $u_\infty=0$, which means that $f(0)=0$. We make the further assumption that $0$ is a stable equilibrium for the kinetics equation associated with (\ref{gen pde}). More precisely, this means that \begin{equation}\label{Turing assumption}
\text{there exists } \beta<0 \text{ such that the } n \text{ eigenvalues } \nu_i \text{ of } QS f'(0) \text{ satisfy } \mathrm{Re }\, \nu_i<\beta.
\end{equation} Among systems of the form (\ref{gen pde}) are activator-inhibitor systems, which are known to support pattern formation. Assumption (\ref{Turing assumption}) is therefore natural, since it is of interest to study the stability of structures which emanate from stable, homogeneous states that are destabilized in the presence of diffusion \cite{Murray,Turing}. Since $\hat{u}_t=0$, the traveling wave equation is an ODE which can be converted to a first order system by introducing the variable $v=S^{-1}u_z$: \begin{equation}\label{tw ode}
\left(\begin{array}{c}
u\\v
\end{array}\right)_z=\left(\begin{array}{c}
Sv\\
-cv-Qf(u)
\end{array}\right).
\end{equation} This perspective is useful, because it opens up the possibility of analyzing (\ref{gen pde}) using dynamical systems techniques. For example, the traveling wave $\vp=(\hat{u},S^{-1}\hat{u})$ is seen to be a homoclinic orbit to the fixed point $(0,0)$. The stable and unstable manifolds of this fixed point--$W^s(0)$ and $W^u(0)$ respectively--will play an important role in our analysis, as will their tangent spaces at $0$: \begin{equation}\label{un/stable subspaces}
V^s(0)=T_0W^s(0),\hspace{.05 in} V^u(0)=T_0W^u(0).
\end{equation}

Our aim is to use the Maslov index to analyze the stability of $\hat{u}$, which is defined as follows. \begin{define}\label{stability defn}
	The traveling wave $\hat{u}(z)$ is \textbf{asymptotically stable} relative to (\ref{tw pde}) if there is a neighborhood $V\subset BU(\bbR,\bbR^n)$ of $\hat{u}(z)$ such that if $u(z,t)$ solves (\ref{tw pde}) with $u(z,0)\in V$, then \[||\hat{u}(z+k)-u(z,t)||_\infty\rightarrow 0 \] as $t\rightarrow\infty$ for some $k\in\bbR$.
\end{define} The stability question immediately leads to the operator \begin{equation}\label{L defn}
L:=\p_z^2+c\p_z+QSf'(\hat{u})
\end{equation} obtained by linearizing the right-hand side of (\ref{tw pde}) around $\hat{u}$. It is known \cite{BJ89,Henry} that the nonlinear stability of $\hat{u}$ (in the sense of Definition \ref{stability defn}) is determined by locating the spectrum of $L$. Assumption (\ref{Turing assumption}) guarantees that the essential spectrum of $L$ is contained in the left-half plane. It therefore suffices to determine whether $L$ has any eigenvalues of positive real part. This is the task for the Maslov index. Before discussing that topic further, we briefly review the history of stability in skew-gradient systems.

The papers \cite{Yan02a,Yan02b} considered standing waves, which are pulses with $c=0$. An instability criterion for these waves was derived in \cite{Yan02b} using an orientation index related to derivatives of the Evans function. Stronger results, akin to those obtained in this work (lower bounds on the Morse index and a stability criterion), were then obtained in \cite{CH14} using the Maslov index. The strategy in that work was to use the index to aid in the calculation of spectral flow \cite{APS} for a family of self-adjoint operators. This calculation relied on a change of variables in the eigenvalue equation that revealed a Hamiltonian structure. A similar change of variables was made in \cite{Jo88} to define and use the Maslov index for standing waves in nonlinear Schr\"{o}dinger equations. An unstable eigenvalue was shown to exist by means of a shooting argument in the manifold of Lagrangian planes. More precisely, a change in the homotopy class of a loop was observed as a (spectral) parameter varied. The existence of the eigenvalue follows since such a change can only occur at an eigenvalue.

It is important to note that in each of the cases mentioned above, the waves considered had zero speed. By contrast, \cite{CJ17} and this work consider traveling waves. This difference is significant, since there is \emph{no} change of variables that makes the eigenvalue equation for $L$ in (\ref{L defn}) Hamiltonian. However, there is a symplectic form for which the set of Lagrangian planes is invariant under the eigenvalue equation; hence the Maslov index can be defined. The trade-off is that self-adjointness of $L$ is lost, so that in general the spectrum will not be real. This spurred the authors of \cite{CJ17} to consider the Evans function $D(\lambda)$ \cite{AGJ,Sandstede02}, and it was shown that the sign of $D'(0)$ is determined by the parity of the Maslov index. On the other hand, the main results of this work are formulated without reference to the Evans function.

Consider $\bbR^{2n}$ endowed with a symplectic form $\omega$. By symplectic, we mean that $\omega$ is nondegenerate, skew-symmetric and bilinear. An $n$-dimensional subspace $V\subset\bbR^{2n}$ is called \emph{Lagrangian} if $\omega(v_1,v_2)=0$ for all $v_{1,2}\in V$. The collection of all such subspaces is clearly a subset of $\Gr_n(\bbR^{2n})$, the Grassmannian of all $n$-dimensional subspaces of $\bbR^{2n}$. In fact, this set is actually a smooth manifold of dimension $n(n+1)/2$, called the Lagrangian Grassmannian and denoted $\Lambda(n)$. It is well-known \cite{Arnold67,Maslov,McDS} that $\pi_1(\Lambda(n))=\mathbb{Z}$ for all $n$, and thus a winding number can be defined for loops in this space. This winding number is the Maslov index. Broadly speaking, the Maslov index counts how many times two paths of Lagrangian subspaces intersect each other. (One of the curves may be fixed, which is the traditional way of defining the index \cite{Maslov}.) We will consider paths that encode the left and right boundary data of potential eigenfunctions for $L$. An intersection therefore corresponds to a function that satisfies both boundary conditions and hence is an eigenfunction.

The rest of this paper is organized as follows. In \S 2, we set up the eigenvalue problem and identify the symplectic structure that makes the analysis possible. In \S 3, the Maslov index is defined, both for a path of Lagrangian planes and for a pair of curves of Lagrangian planes. This includes a careful consideration of the ``crossing form'' of \cite{RS93}. In \S 4, we introduce the ``Maslov box" of \cite{HLS16} and show how the Maslov index can be used to give a lower bound on the number of unstable eigenvalues for $L$ in (\ref{L defn}). We apply the same framework to a FitzHugh-Nagumo system in \S 5 and show how the Maslov index gives an exact count of the positive, unstable eigenvalues in this case. Additionally, we prove that any unstable spectrum must be real, from which it follows that the Maslov index detects all unstable eigenvalues. Finally, in \S 6 we show what the Maslov index reveals about the algebraic multiplicity of eigenvalues. This is accomplished by relating the crossing form to derivatives of the Evans function. In particular, we provide a new geometric interpretation of simplicity of an eigenvalue.

\section{Eigenvalue Equation and Symplectic Structure}
As noted above, the stability of $\hat{u}$ is assessed by determining the spectrum $\sigma(L)$ of the operator $L$ in (\ref{L defn}). First, we say that $\lambda\in\bbC$ is an eigenvalue for $L$ if there exists a solution $p\in BU(\bbR,\bbC^n)$ to the equation
\begin{equation}\label{eval problem gen}
Lp=\lambda p.
\end{equation}
The set of isolated eigenvalues of $L$ of finite multiplicity is denoted $\sigma_n(L)$. Comparing with (\ref{L defn}), setting $p_z=Sq$ converts (\ref{eval problem gen}) to the first order system \begin{equation}\label{eval eqn matrix}
\left(\begin{array}{c}
p\\q
\end{array}\right)'=\left(\begin{array}{c c}
0 & S\\
\lambda S^{-1}-Qf'(\hat{u}) & -cI
\end{array} \right)\left(\begin{array}{c}
p\\q
\end{array}\right).
\end{equation} As is commonly done for Evans function analyses (see \cite{AGJ}), we abbreviate (\ref{eval eqn matrix}) as \begin{equation}
Y'(z)=A(\lambda,z)Y(z),
\end{equation} with $Y\in \bbC^n$ and $A(\lambda,z)\in M_n(\bbC^{2n})$. Assumption (\ref{Turing assumption}) guarantees that $\hat{u}$ approaches $0$ exponentially, and thus there is a well-defined matrix \begin{equation}
A_\infty(\lambda)=\lim\limits_{z\rightarrow\pm\infty}A(\lambda,z),
\end{equation} and this limit is also achieved exponentially quickly. The eigenvalues of $L$ comprise only part of the spectrum; the rest is the essential spectrum $\sigma_\mathrm{ess}(L)$. For systems of the form (\ref{gen pde}), it is known (Lemma 3.1.10 of \cite{KP13}) that the essential spectrum is given by \begin{equation}
\sigma_\mathrm{ess}(L)=\{\lambda\in\bbC:A_\infty(\lambda) \text{ has an eigenvalue }\mu\in i\bbR \}.
\end{equation}

We claim that $\sigma_\mathrm{ess}(L)$ is contained in the half-plane \begin{equation}
H=\{\lambda\in\bbC:\mathrm{Re}\,\lambda<\beta\}.
\end{equation}Indeed, a simple calculation using (\ref{eval eqn matrix}) shows that the eigenvalues of $A_\infty(\lambda)$ are given by \begin{equation}\label{evals of A(lambda)}
\mu_j(\lambda)=\frac{1}{2}\left(-c\pm\sqrt{c^2+4(\lambda-\nu_i)}\right),
\end{equation} with $\nu_i$ from (\ref{Turing assumption}). We need to show that $A_\infty(\lambda)$ has no purely imaginary eigenvalues if $\mathrm{Re } \lambda\geq\beta$, which is clearly equivalent to showing that $\mathrm{Re }\sqrt{c^2+4(\lambda-\nu_i)}\neq-c$ for such $\lambda$. The formula \begin{equation}\label{sqrt ineq}
\mathrm{Re }\sqrt{a+bi}=\frac{1}{\sqrt{2}}\sqrt{\sqrt{a^2+b^2}+a}
\end{equation} and the fact that $\mathrm{Re}\,(c^2+4(\lambda-\nu_i))>0$ from (\ref{Turing assumption}) together imply that \begin{equation}
\mathrm{Re}\sqrt{c^2+4(\lambda-\nu_i)}\geq\sqrt{\mathrm{Re}(c^2+4(\lambda-\nu_i))}>\sqrt{c^2}=-c,
\end{equation} as desired. This calculation actually proves that $A_\infty(\lambda)$ has exactly $n$ eigenvalues of positive real part and $n$ eigenvalues of negative real part for $\lambda\in (\bbC\setminus H).$ We label these $\mu_i(\lambda)$ in order of increasing real part and observe that \begin{equation}\label{evals inequality}
\mathrm{Re}\,\mu_1(\lambda)\leq \dots\leq\mathrm{Re}\,\mu_n(\lambda)<0<-c<\mathrm{Re}\,\mu_{n+1}(\lambda)\leq\dots\leq\mathrm{Re}\,\mu_{2n}(\lambda).
\end{equation} Furthermore, one sees from (\ref{evals of A(lambda)}) that for each $1\leq i\leq n$ we have \begin{equation}
\mu_i(\lambda)+\mu_{i+n}(\lambda)=-c.
\end{equation} It then follows from, for example, Theorem 3.2 of \cite{Sandstede02} that (\ref{eval eqn matrix}) has exponential dichotomies on $\bbR^+$ and $\bbR^-$ for $\lambda\in\bbC\setminus H$, allowing us to define $n$-dimensional vector spaces \begin{equation}\label{(un)stable bundles}
 \begin{aligned}
 E^u(\lambda,z) & = \{\xi(z)\in\bbC^{2n}:\xi \text{ solves } (\ref{eval eqn matrix}) \text{ and } \xi\rightarrow 0 \text{ as }z\rightarrow -\infty \}\\
 E^s(\lambda,z) & = \{\xi(z)\in\bbC^{2n}:\xi \text{ solves } (\ref{eval eqn matrix}) \text{ and } \xi\rightarrow 0 \text{ as }z\rightarrow \infty \}
 \end{aligned}.
 \end{equation} We call these sets the \emph{unstable} and \emph{stable bundles} respectively. It is known that $E^{s/u}(\lambda,z)$ vary analytically in $\lambda$ for each $z$. Moreover, the decay of the solutions in $E^{u/s}(\lambda,z)$ is exponential, and any solution of (\ref{eval eqn matrix}) that is bounded at $-\infty$ (resp. $\infty$) must be a member of $E^u(\lambda,z)$ (resp. $E^s(\lambda,z))$. It follows that $\lambda\in\bbC$ is an eigenvalue for $L$ if and only if the intersection $E^u(\lambda,z)\cap E^s(\lambda,z)$ is nonempty for some (and hence all) $z\in\bbR$. The fact that any eigenfunction of $L$ must decay exponentially allows us instead to pose the eigenvalue problem on the Hilbert space $H^1(\bbR,\bbC^n)$. This will pay dividends later when we consider the FitzHugh-Nagumo system.
 
We now focus our attention on real $\lambda\geq\beta$. In this case, $E^{s/u}(\lambda,z)$ are real vector spaces. To identify the symplectic structure, we introduce the matrix \begin{equation}\label{complex structure}
J=\left(\begin{array}{c c}
0 & Q\\
-Q & 0
\end{array}\right).
\end{equation} Since $Q^2=I$ and $Q^*=Q$, it follows that $J^2=-I$ and $J^*=-J$. We therefore call $J$ a complex structure on $\bbR^{2n}$. If we denote by $\langle\cdot,\cdot\rangle$ the standard inner product on $\bbR^{2n}$, then \begin{equation}\label{symplectic form compatible}
\omega(a,b):=\langle a,Jb\rangle
\end{equation} defines a symplectic form on $\bbR^{2n}$, see \S 1 of \cite{Heck13}, for example. The following theorem underpins all of the ensuing analysis.

\begin{theorem}\label{form invariance thm}
Let $Y_1$, $Y_2$ be any two solutions of (\ref{eval eqn matrix}) for fixed $\lambda\in\bbR$. Then \begin{equation}
\frac{d}{dz}\omega(Y_1,Y_2)=-c\,\omega(Y_1,Y_2).
\end{equation} In particular, if $\omega(Y_1(z_0),Y_2(z_0))=0$ for some $z_0\in\bbR$, then $\omega(Y_1,Y_2)\equiv 0$. More generally, the symplectic form \begin{equation}\label{Omega defn}
\Omega:=e^{cz}\omega
\end{equation} is constant in $z$ on any two solutions of (\ref{eval eqn matrix}).
\end{theorem}

\begin{proof}
A direct computation gives that \begin{equation}
\begin{aligned}
\frac{d}{dz}\omega(Y_1,Y_2) & =\omega(Y_1,A(\lambda,z)Y_2)+\omega(A(\lambda,z)Y_1,Y_2)\\
& = \langle Y_1,JA(\lambda,z)Y_2\rangle+\langle A(\lambda,z)Y_1,JY_2\rangle\\
& = \langle Y_1,\left[JA+A^TJ\right]Y_2\rangle. 
\end{aligned}
\end{equation} In light of (\ref{symplectic form compatible}), we therefore need to show that \begin{equation}\label{compatibility calc}
JA+A^TJ=-cJ.
\end{equation} Recalling that $S$ and $Q$ are diagonal and that  $(f'(\hat{u}))^T=F''(\hat{u})^T=F''(\hat{u})=f'(\hat{u})$, we compute \begin{equation}
\begin{aligned}
JA+A^TJ & =\left(\begin{array}{c c}
\lambda QS^{-1}-f'(\hat{u}) & -cQ\\
0 & -QS
\end{array}\right)+\left(\begin{array}{c c}
-\lambda S^{-1}Q+f'(\hat{u}) & 0\\
cQ & SQ
\end{array}\right)\\
& = -c\left(\begin{array}{c c}
0 & Q\\
-Q & 0
\end{array}\right)=-cJ.
\end{aligned}
\end{equation} For the second part, we see that \begin{equation}
\frac{d}{dz}\Omega(Y_1,Y_2)=e^{cz}\left(c\omega(Y_1,Y_2)+\frac{d}{dz}\omega(Y_1,Y_2)\right)=0.
\end{equation}
\end{proof}

For fixed $\lambda\in\bbR$, it is a standard result that (\ref{eval eqn matrix}) respects linear independence of solutions. It follows that (\ref{eval eqn matrix}) induces a flow on $\Gr_k(\bbR^{2n})$ for any $k$. The following is then a consequence of the preceding theorem.
\begin{cor}\label{L(n) invariant}
The set of $\omega$-Lagrangian planes $\Lambda(n)$ is an invariant manifold for the equation induced by (\ref{eval eqn matrix}) on $\Gr_n(\bbR^{2n})$.
\end{cor} 

As explained above, eigenvalues are found by looking for intersections of the sets $E^{s/u}(\lambda,z)$. To make use of Corollary \ref{L(n) invariant}, it is therefore critical that the stable and unstable bundles are actually Lagrangian. We show now that this is indeed the case.

\begin{theorem}\label{bundles Lag thm}
For all $\lambda\in\bbR\cap(\bbC-H)$ and $z\in\bbR$, the subspaces $E^u(\lambda,z)$ and $E^s(\lambda,z)$ are Lagrangian.
\end{theorem}

\begin{proof}
First, it is clear that $\omega$ and $\Omega$ define the same set of Lagrangian planes. By Theorem \ref{form invariance thm}, we just need to show that $\Omega(Y_1,Y_2)=0$ for some value of $z$, given $Y_1,Y_2\in E^{s/u}(\lambda,z)$. We begin with $E^s(\lambda,z)$. By definition, $Y_1,Y_2\in E^s(\lambda,z)$ decay to $0$ as $z\rightarrow\infty$. Since $c<0$, it is easy to see that \begin{equation}
\Omega(Y_1,Y_2)=\lim\limits_{z\rightarrow\infty}e^{cz}\omega(Y_1,Y_2)=0.
\end{equation} Now consider $Y_1,Y_2\in E^u(\lambda,z)$. The decay of these solutions at $-\infty$ will be faster than $e^{-cz}$, by (\ref{evals inequality}) and Theorem 3.1 of \cite{Sandstede02}. It follows that \begin{equation}
\Omega(Y_1,Y_2)=\lim\limits_{z\rightarrow-\infty}e^{cz}\omega(Y_1,Y_2)=\lim\limits_{z\rightarrow-\infty}\omega(e^{cz}Y_1,Y_2)=0.
\end{equation} This completes the proof.
\end{proof}

The result of this section is that the stable and unstable bundles define smooth two-parameter curves in $\Lambda(n)$--a lower-dimensional submanifold of $\Gr_n(\bbR^{2n})$. The way to exploit this fact is through the Maslov index, which we discuss in the next section. We close this section by pointing out that the systems we consider are less general than the ``skew-gradient" systems of \cite{Yan02a,Yan02b,CH14}, which allow for a positive, diagonal matrix $D$ to multiply $u_{xx}$ in (\ref{gen pde}). The reason for this is that the proof of Theorem \ref{form invariance thm} breaks down if we include the matrix $D$ due to the presence of the convective term $cu_z$. It is quite interesting that we are free to control the coupling of the terms in the reaction (through the matrix $S$), but that changing the diffusivities of the reagents ruins the symplectic structure.
\section{The Maslov Index}
As mentioned in the introduction, the fundamental group of $\Lambda(n)$ is infinite cyclic for all $n\in\bbN$. The homotopy class of a loop in this space is therefore like a winding number. Intuitively, the duality between winding numbers and intersection numbers should allow us to identify the homotopy class of a loop as an intersection count with a codimension one set in $\Lambda(n)$. This is indeed the case, as was shown by Arnol'd \cite{Arnold67}. In fact, Arnol'd extended this definition to non-closed curves under certain assumptions. These assumptions were relaxed considerably in \cite{RS93}, and the intersection number discussed therein is the Maslov index that we will employ.

To start, fix a Lagrangian plane $V\in\Lambda(n)$ and define the \emph{train} of $V$ to be \begin{equation}\label{train defn}
\Sigma(V)=\{V'\in\Lambda(n):V\cap V'\neq\{0\} \}.
\end{equation} There is a natural partition of this set into submanifolds of $\Lambda(n)$ given by \begin{equation}
\Sigma(V)=\bigcup_{k=1}^n\Sigma_k(V),\hspace{.2 in}\Sigma_k(V)=\{V'\in\Lambda(n):\dim(V\cap V')=k \}.
\end{equation} In particular, the set $\Sigma_1(V)$ is dense in $\Sigma(V)$, and it is a two-sided, codimension one submanifold of $\Lambda(n)$ (cf. \S 2 of \cite{RS93}). In \cite{Arnold67}, the Maslov index of a loop $\alpha$ is defined as the number of signed intersections of $\alpha$  with $\Sigma_1(V)$. A homotopy argument is used to ensure that all intersections with $\Sigma(V)$ are actually with $\Sigma_1(V)$, and hence this definition makes sense. More generally, for a curve $\gamma:[a,b]\rightarrow\Lambda(n)$, it is shown (\S 2.2 of \cite{Arnold67}) that the same index is well-defined, provided that $\gamma(a),\gamma(b)\notin\Sigma(V)$ and that all intersections with $\Sigma(V)$ are one-dimensional and transverse. Both the assumptions of transversality at the endpoints and of only one-dimensional crossings were dispensed of in \cite{RS93}. The key was to make robust the notion of intersections with $\Sigma(V)$, which was accomplished through the introduction of the ``crossing form.''

Now let $\gamma:[a,b]\rightarrow\Lambda(n)$ be a smooth curve. The tangent space to $\Lambda(n)$ at any point $\gamma(t)$ can be identified with the space of quadratic forms on $\gamma(t)$ (cf. \S 1.6 of \cite{Duis}). This allows one to define a quadratic form that determines whether $\gamma(t)$ is transverse to $\Sigma(V)$ at a given intersection; this quadratic form is the crossing form. Specifically, suppose that $\gamma(t^*)\in\Sigma(V)$ for some $t^*\in[a,b]$. It can be checked from (\ref{symplectic form compatible}) that the plane $J\cdot\gamma(t^*)$ is orthogonal to $\gamma(t^*)$, with $J$ as in (\ref{complex structure}). Furthermore, any other Lagrangian plane $W$ transverse to $J\cdot\gamma(t^*)$ can be written uniquely as the graph of a linear operator $B_W:\gamma(t^*)\rightarrow J\cdot\gamma(t^*)$ \cite{Duis}. This includes $\gamma(t)$ for $|t-t^*|<\delta\ll 1.$ Writing $B_{\gamma(t)}=B(t)$, it follows that the curve $v+B(t)v\in\gamma(t)$ for all $v\in\gamma(t^*)$. The crossing form is then defined by \begin{equation}\label{crossing form RS}
\Gamma(\gamma,V,t^*)(v)=\frac{d}{dt}\omega(v,B(t)v)|_{t=t^*}.
\end{equation} The form is defined on the intersection $\gamma(t^*)\cap V$. It is shown in Theorem 1.1 of \cite{RS93} that this definition is independent of the choice $J\cdot\gamma(t^*)$; any other \emph{Lagrangian complement} of $\gamma(t^*)$ would produce the same crossing form. The crossing form is quadratic, so it has a well-defined signature. For a quadratic form $Q$, we use the notation $\mathrm{sign}(Q)$ for its signature. We also write $n_+(Q)$ and $n_-(Q)$ for the positive and negative indices of inertia of $Q$ (see page 187 of \cite{vinberg}), so that \begin{equation}
\mathrm{sign}(Q)=n_+(Q)-n_-(Q).
\end{equation}Roughly speaking, $\mathrm{sign}(\Gamma(\gamma,V,t^*))$ gives the dimension and the direction of the intersection $\gamma(t^*)\cap V$. Reminiscent of Morse theory, a value $t^*$ such that $\gamma(t^*)\cap V\neq\{0\}$ is called a \emph{conjugate point} or \emph{crossing}. A crossing is called \emph{regular} if the associated form $\Gamma$ is nondegenerate. One can then define the Maslov index as follows. \begin{define}\label{Maslov defn num}
Let $\gamma:[a,b]\rightarrow \Lambda(n)$ and $V\in\Lambda(n)$ such that $\gamma(t)$ has only regular crossings with the train of $V$. The \emph{Maslov index} is then given by \begin{equation}\label{Maslov defn}
	\mu(\gamma,V)=-n_-(\Gamma(\gamma,V,a))+\sum\limits_{t^*\in(a,b)}\mathrm{sign}\,\Gamma(\gamma,V,t^*)+n_+(\Gamma(\gamma,V,b)),
	\end{equation} where the sum is taken over all interior conjugate points.
\end{define}
\begin{rem}
	The reader will notice that the Maslov index defined in \cite{RS93} has a different endpoint convention than Definition \ref{Maslov defn num}. Instead, they take $(1/2)\mathrm{sign}(\Gamma)$ as the contribution at both $a$ and $b$. This is merely convention, provided that one is careful to make sure that the additivity property (see Proposition \ref{Maslov props} below) holds. Our convention follows \cite{CJ17,HLS16} to make sure that the Maslov index is always an integer.
\end{rem}
The convention on the endpoints in (\ref{Maslov defn}) serves to ensure that the Maslov index has (among others) the following nice properties from \S 2 of \cite{RS93}.
\begin{prop}\label{Maslov props} 
Let $\gamma:[a,b]\rightarrow\Lambda(n)$ be a curve with only regular crossings. Then \begin{enumerate}[label=(\roman*)]
		\item (Additivity by concatenation) For any $c\in(a,b)$, $\mu(\gamma,V)=\mu(\gamma|_{[a,c]},V)+\mu(\gamma|_{[c,b]},V)$.
		\vspace{.05 in}
		\item (Homotopy invariance) Two paths $\gamma_{1,2}:[a,b]\rightarrow\Lambda(2)$ with $\gamma_1(a)=\gamma_2(a)$ and $\gamma_1(b)=\gamma_2(b)$ are homotopic with fixed endpoints if and only if $\mu(\gamma_1,V)=\mu(\gamma_2,V)$.
		\vspace{.05 in}
		\item If $\dim(\gamma(t)\cap V)= \text{ constant}$, then $\mu(\gamma,V)=0$.
	\end{enumerate}
\end{prop}

Up until now, we have considered one curve of Lagrangian planes and seen how to count intersections with a fixed reference plane. Alternatively, one could consider two curves of Lagrangian planes and count how many times they intersect each other. This theory is also developed in \cite{RS93}, see \S 3. Suppose then that we have two curves $\gamma_{1,2}:[a,b]\rightarrow\Lambda(n)$. If $\gamma_1(t^*)\cap\gamma_2(t^*)\neq\{0\}$ for some $t^*\in[a,b],$ then we can define the \emph{relative crossing form} \begin{equation}\label{rel crossing form}
\Gamma(\gamma_1,\gamma_2,t^*)=\Gamma(\gamma_1,\gamma_2(t^*),t^*)-\Gamma(\gamma_2,\gamma_1(t^*),t^*)
\end{equation} on the intersection $\gamma_1(t^*)\cap\gamma_2(t^*)$. As before, a crossing is regular if $\Gamma$ in (\ref{rel crossing form}) is nondegenerate. For two curves with only regular crossings, define the \emph{relative Maslov index} to be \begin{equation}\label{rel Maslov index}
\mu(\gamma_1,\gamma_2)=-n_-(\Gamma(\gamma_1,\gamma_2,a))+\sum\limits_{t^*\in(a,b)}\mathrm{sign}\,\Gamma(\gamma_1,\gamma_2,t^*)+n_+(\Gamma(\gamma_1,\gamma_2,b)),
\end{equation} where again the sum is taken over interior intersections of $\gamma_1$ and $\gamma_2$. We point out now that it is easy to show that regular crossings are isolated, so the sums in both (\ref{Maslov defn}) and (\ref{rel Maslov index}) are finite. Also, it is clear from (\ref{rel crossing form}) that (\ref{rel Maslov index}) coincides with Definition \ref{Maslov defn num} in the case where $\gamma_2=$ constant. Accordingly, most of the properties of the Maslov index in \S 2 of \cite{RS93} carry over to the two-curve case without much trouble. However, in moving from paths to pairs of curves, one must be careful about the homotopy axiom. The following is proved in Corollary 3.3 of \cite{RS93}.
\begin{prop}\label{homotopy rel Mas}
	Let $\gamma_1$ and $\gamma_2$ be curves of Lagrangian planes with common domain $[a,b]$. If $\gamma_1(a)\cap\gamma_2(a)=\gamma_1(b)\cap\gamma_2(b)=\{0\}$, then $\mu(\gamma_1,\gamma_2)$ is a homotopy invariant, provided that the homotopy respects the stated condition on the endpoints.
\end{prop}

We are now ready to specialize to the problem at hand. Recall that for $\lambda\geq\beta$, $E^u(\lambda,z)$ and $E^s(\lambda,z)$ are both members of $\Lambda(n)$ for all $z\in\bbR$. We will relate the Maslov index to eigenvalues of $L$ by looking for intersections of these subspaces as $z$ and $\lambda$ vary. Thus there are two practical formulations of (\ref{crossing form RS}) for our purposes: one for curves parametrized by $z$ and another for curves parametrized by $\lambda$. For curves parametrized by $z$, we have the following formula, which is proved as Theorem 3 of \cite{CJ17}.

\begin{theorem}\label{crossing form z}
	Consider the curve $z\mapsto E^u(\lambda,z)$, for fixed $\lambda$. Assume that for a reference plane $V$, there exists a value $z=z^*$ such that $E^u(\lambda,z^*)\cap V\neq\{0\}$. Then the crossing form for $E^u(\lambda,\cdot)$ with respect to $V$ is given by \begin{equation}\label{crossing form z formula}
	\Gamma(E^u(\lambda,\cdot),V,z^*)(\zeta)=\omega(\zeta,A(\lambda,z^*)\zeta),
	\end{equation} restricted to the intersection $E^u(\lambda,z^*)\cap V$.
\end{theorem}

We postpone deriving the $\lambda$-crossing form until \S 5. For now, we return to the motivation of this project, which is to use features of the wave $\hat{u}$ (or $\varphi$) itself to determine its stability. More precisely, we want to associate a Maslov index to $\varphi$ that we can calculate and use to say something about the unstable spectrum of $L$. The Maslov index of a homoclinic orbit was defined in a rigorous way in \cite{CH}. The curve of Lagrangian planes is given by $z\mapsto E^u(0,z)$, which can be thought of as the space of solutions to (\ref{eval eqn matrix}) with $\lambda=0$ satisfying the `left' boundary condition. In the spirit of a shooting argument, the natural choice of reference plane is $V^s(0)$, the stable subspace of the linearization of (\ref{tw ode}) about $0$. However, for technical reasons this is untenable. Indeed, by translation invariance, the derivative of the wave $\vp'(z)\in E^u(0,z).$ Since a homoclinic orbit approaches its end state tangent to the stable manifold, there would be a conjugate point at $+\infty$. Moreover, this conjugate point would be irregular since it is reached in infinite time.

To address this issue, Chen and Hu (\S 1 of \cite{CH}) instead pulled back $V^s(0)$ slightly along $\vp$ and used $E^s(0,\tau)$, $\tau\gg 1$ as a reference plane. The domain of the curve is truncated as well to $(-\infty,\tau]$, which forces a conjugate point at the right end point; $\vp'$ (at least) is in the intersection $E^u(0,\tau)\cap E^s(0,\tau)$. The only requirement on $\tau$ is that \begin{equation}\label{tau req}
V^u(0)\cap E^s(0,z)=\{0\} \text{ for all } z\geq\tau.
\end{equation}  One then arrives at the following definition. \begin{define}\label{Maslov of phi defn num}
	Let $\tau$ satisfy (\ref{tau req}). The \textbf{Maslov index} of $\varphi$ is given by \begin{equation}\label{Maslov phi defn}
	\mathrm{Maslov}(\varphi):=\sum_{z^*\in(-\infty,\tau)}\mathrm{sign}\,\Gamma(E^u,E^s(0,\tau),z^*)+n_+(\Gamma(E^u,E^s(0,\tau),\tau)),
	\end{equation} where the sum is taken over all interior crossings of $E^u(0,z)$ with $\Sigma$, the train of $E^s(0,\tau)$.
\end{define}

It was shown in \cite{CH} that this definition is independent of $\tau$, as long as (\ref{tau req}) is satisfied. This is important, because we will have to revise the value $\tau$ to complete the arguments of the next section.

Although we consider several curves in this work, any mention of \emph{the} Maslov index is referring to $\Maslov$. The value $\lambda=0$ is special because (\ref{eval eqn matrix}) is the equation of variations for (\ref{tw ode}) in that case. Accordingly, one can show (cf. \S 6 of \cite{CJ17}) that $E^u(0,z)$ is tangent to $W^u(0)$ along $\vp$. The Maslov index can therefore be interpreted as the number of twists $W^u(0)$ makes as $\vp$ moves through phase space. In the case $n=1$, $E^u(0)$ is spanned by the velocity to the wave, and conjugate points correspond to zeros of $\vp'(z)$ (albeit with a rotation of the reference plane). In this way, one sees that the Maslov index can be used to derive Sturm-Liouville theory (see also \S 1 of \cite{BCJLMS17}).

\section{The Maslov Box} 
We will see in this section that the set of positive, real eigenvalues of $L$ is bounded above. Since the spectrum of $L$ in $\bbC\setminus H$ consists of isolated eigenvalues of finite multiplicity (cf. page 172 of \cite{AGJ}), it follows that the quantity \begin{equation}\label{Morse-Maslov ineq}
\Mor(L):= \text{ the number of real, positive eigenvalues of } L \text{ counting algebraic multiplicity}
\end{equation} is well defined. The rest of this section is dedicated to proving \begin{theorem}\begin{equation}
	|\Maslov|\leq \Mor(L).
	\end{equation}
\end{theorem}
The strategy of the proof is to consider a contractible loop in $\Lambda(n)\times\Lambda(n)$ (the ``Maslov box'') consisting of four different curve segments. Since the total Maslov index must be zero, Proposition \ref{Maslov props}(i) guarantees that the sum of the constituent Maslov indices is zero. Two of these segments have Maslov index zero, one of them is $\Maslov$, and the final segment is bounded above by $\Mor(L)$. This strategy has its roots in \cite{JLM13,HS16,HLS16}. In particular, \cite{HLS16} coined the term ``Maslov box," and that paper encounters many of the same difficulties that arise when considering homoclinic orbits (i.e. curves on infinite intervals). The difference between this paper and \cite{HLS16} is that the latter considered gradient reaction-diffusion equations. In that case, the linearized operator $L$ is self-adjoint, and the Maslov index is computed using spectral flow of unitary matrices.

From this point forward, we will think of the stable and unstable bundles as curves in $\Lambda(n)$. Likewise, we think of the stable and unstable subspaces of $A_\infty(\lambda)$ as points in $\Gr_n(\bbR^{2n})$. We call these spaces $S(\lambda)$ and $U(\lambda)$ respectively. (In particular, $V^s(0)=S(0)$ and $V^u(0)=U(0)$.) It follows from Lemma 3.2 of \cite{AGJ} that \begin{equation}
\begin{aligned}
\lim\limits_{z\rightarrow-\infty}E^u(\lambda,z)=U(\lambda)\\
\lim\limits_{z\rightarrow\infty}E^s(\lambda,z)=S(\lambda).
\end{aligned}
\end{equation} In light of Theorem \ref{bundles Lag thm}, this actually proves that \begin{equation}
S(\lambda),U(\lambda)\in\Lambda(n),
\end{equation} since $\Lambda(n)$ is a closed submanifold of $\Gr_n(\bbR^{2n})$. In what follows, it will be important to know what happens to $E^u(\lambda,z)$ as $z\rightarrow\infty$. First, if $\lambda\in\sigma(L)$, then $E^u(\lambda,z)\cap E^s(\lambda,z)\neq\{0\}$, so it must be the case that \begin{equation}\label{+ limit eval}
\lim\limits_{z\rightarrow\infty} E^u(\lambda,z)\in \Sigma(S(\lambda)),
\end{equation} the train of $S(\lambda)$. On the other hand, if $\lambda\notin\sigma(L)$, then any solution of (\ref{eval eqn matrix}) is unbounded at $+\infty$, and it is proved in Lemma 3.7 of \cite{AGJ} that \begin{equation}\label{+ limit no eval}
\lim\limits_{z\rightarrow\infty}E^u(\lambda,z)=U(\lambda).
\end{equation}
There are a few facts to be gleaned from this observation. First, if $\lambda\notin\sigma(L)$, then $z\mapsto E^u(\lambda,z)$ (with domain $\bbR$) forms a loop in $\Lambda(n)$, in which case the Maslov index is independent of the choice of reference plane (\S 1.5 of \cite{Arnold67}). This fact was used in \cite{Jo88}, which is the first appearance of the Maslov index for solitary waves (known to the author). Also, it follows from (\ref{+ limit eval}) and (\ref{+ limit no eval}) that \begin{equation}
\lim\limits_{z\rightarrow\infty} E^u(\lambda,z)
\end{equation} is discontinuous in $\lambda$ at each eigenvalue of $L$. Indeed, $U(\lambda)$ is bounded away from $\Sigma(S(\lambda))$, since $\bbR^{2n}=S(\lambda)\oplus U(\lambda)$. This is the motivation for using the cutoff $x_\infty$ (or $\tau$ in this paper) for the unstable bundle in \cite{HLS16}, since the homotopy argument requires a continuous curve. We also have the additional motivation for the cutoff of using $\Maslov$ explicitly. 

Proposition 2.2 of \cite{AGJ} guarantees that one can draw a simple, closed curve in $\bbC$ containing $\sigma(L)\cap(\bbC\setminus H)$ in its interior. An obvious consequence of this is that the real, unstable spectrum of $L$ is bounded above by a constant $M$. We will make use of the following, slightly stronger fact.
\begin{lemma}\label{right shelf lemma}
	 There exists $\lambda_\mathrm{max}>M$ such that, for all $z\in\bbR$, \begin{equation} 
	E^u(\lambda_\mathrm{max},z)\cap S(\lambda_\mathrm{max})=\{0\}.
	\end{equation}
\end{lemma}

This is proved in \S 4.5 of \cite{HLS16}, and we refer the reader there for the details. We will instead outline the basic idea, which is straightforward. System (\ref{eval eqn matrix}) is a perturbation of the autonomous system $Y'(z)=A_\infty(\lambda)Y(z)$, which also induces an equation on $\Gr_n(\bbR^{2n})$. In this latter system, $U(\lambda)$ is an attracting fixed point, as is observed in the proof of Lemma 3.7 in \cite{AGJ}. We can therefore find a small ball $B$ around $U(\lambda)$ in $\Lambda(n)$ on the boundary of which the vector field points inward. Furthermore, this ball can be taken small enough to be disjoint from $\Sigma(S(\lambda))$, which is itself closed in $\Lambda(n)$. For large enough $\lambda$, (\ref{eval eqn matrix}) is essentially autonomous, so $B$ will still be positively invariant. Finally, since any $\lambda>M$ is not an eigenvalue of $L$, the curve $z\mapsto E^u(\lambda,z)$ will both emanate from and return to $U(\lambda)$. It will therefore be trapped in the ball $B$, and hence there will be no intersections with $S(\lambda)$.

Now fix the value $\lambda_\mathrm{max}$ guaranteed by the preceding lemma. The immediate goal is to set, once and for all, the value $\tau$ appearing in Definition \ref{Maslov defn num}. Since $E^s(\lambda_\mathrm{max},z)\rightarrow S(\lambda_\mathrm{max})$ as $z\rightarrow\infty$, it follows from Lemma \ref{right shelf lemma} that we can find a value $z=\tau_\mathrm{max}$ such that \begin{equation}
E^u(\lambda_\mathrm{max},z)\cap E^s(\lambda_\mathrm{max},\zeta)=\{0\}, \hspace{.1 in} \text{ for all } z\in\bbR \text{ and for all }\zeta\geq\tau_\mathrm{max}.
\end{equation} Similarly, for each $\lambda\in[0,\lambda_\mathrm{max}]$ we can find $\tau_\lambda$ and an open interval $I_\lambda$ containing $\lambda$ such that \begin{equation}
U(\lambda)\cap E^s(\lambda,z)=\{0\},\hspace{.1 in} \text{ for all } z\geq\tau_\lambda, \lambda\in I_\lambda.
\end{equation} (For $\lambda=\lambda_\mathrm{max}$, the value $\tau_\mathrm{max}$ defined above works just fine.) Extracting a finite subcover $\cup_{k=1}^NI_{\lambda_k}$ of $[0,\lambda_\mathrm{max}]$, we set \begin{equation}\label{tau defn}
\tau=\max\{\tau_{\lambda_1},\dots,\tau_{\lambda_k}\}.
\end{equation} The preceding can be summarized in the following proposition.
\begin{prop}\label{shelves prop}
	With $\tau$ given by (\ref{tau defn}), the following are true. \begin{enumerate}[label=(\roman*)]
		\item $E^u(\lambda_\mathrm{max},z)\cap E^s(\lambda_\mathrm{max},\tau)=\{0\}$ for all $z\in(-\infty,\tau]$.
		\vspace{.05 in}
		\item $U(\lambda)\cap E^s(\lambda,\tau)=\{0\}$ for all $\lambda\in[0,\lambda_\mathrm{max}].$
	\end{enumerate}
\end{prop}

Now consider the rectangle \begin{equation}\label{rectangle}
Q=[0,\lambda_\mathrm{max}]\times[-\infty,\tau].
\end{equation} $Q$ is mapped into $\Lambda(n)\times\Lambda(n)$ by the function \begin{equation}\label{G defn}
G(\lambda,z)=(E^u(\lambda,z),E^s(\lambda,\tau)),
\end{equation} where $G(\lambda,-\infty)$ is defined to be $(U(\lambda),E^s(\lambda,\tau))$. Notice that $G$ is continuous, see \S 3 of \cite{AGJ}. Since $Q$ is contractible, the image $G(Q)\subset\Lambda(2)\times\Lambda(2)$ is contractible as well. Let $F:Q\times[0,1]\rightarrow Q$ be a deformation retract (page 361 of \cite{Munkres}) of $Q$ onto the point $(0,-\infty)$. Composing $F$ and $G$ then gives a deformation retract of $G(Q)$ onto $G(0,-\infty)=(U(0),E^s(0,\tau))$. In particular, we see that the image of the boundary $\p Q$ (with a counterclockwise orientation) under $G$ is homotopic with fixed endpoints to the constant path $(U(0),E^s(0,\tau))$. We will call this (closed) boundary curve $\alpha$. Since $U(0)\cap E^s(0,\tau)=\{0\}$ by Proposition \ref{shelves prop}(ii), we see that Proposition \ref{homotopy rel Mas} applies, so \begin{equation}
\mu(\alpha)=\mu(U(0),E^s(0,\tau))=0.
\end{equation} The Maslov index in this case is for pairs of Lagrangian planes, since $\alpha\subset\Lambda(n)\times\Lambda(n)$. We can describe the loop $\alpha$ as the concatenation of four curve segments. (See Figure \ref{fig1} below.) Define:
 \begin{equation}\label{alpha defn}
\begin{aligned}
\alpha_1 & = (E^u(0,z),E^s(0,\tau)), \hspace{.1 in} z\in[-\infty,\tau]\\
\alpha_2 & = (E^u(\lambda,\tau),E^s(\lambda,\tau)), \hspace{.1 in} \lambda\in[0,\lambda_\mathrm{max}]\\
\alpha_3 & = (E^u(\lambda_\mathrm{max},-z),E^s(\lambda_\mathrm{max},\tau)), \hspace{.1 in} z\in[-\tau,\infty] \\
\alpha_4 & =(U(\lambda_\mathrm{max}-\lambda),E^s(\lambda_\mathrm{max}-\lambda,\tau)), \hspace{.1 in} \lambda\in[0,\lambda_\mathrm{max}].
\end{aligned}
\end{equation}

Using the notation of \cite{Munkres}, page 326, it is clear that $\alpha=\alpha_1*\alpha_2*\alpha_3*\alpha_4$. As explained above, $\mu(\alpha)=0$, since $G(Q)$ is contractible. Proposition \ref{Maslov props}(i) then asserts that \begin{equation}\label{additivity equation}
0=\mu(\alpha)=\mu(\alpha_1)+\mu(\alpha_2)+\mu(\alpha_3)+\mu(\alpha_4).
\end{equation}
It is a direct consequence of Proposition \ref{shelves prop}(i) that $\mu(\alpha_3)=0$, since there are no conjugate points. Likewise, Proposition \ref{shelves prop}(ii) says that $\mu(\alpha_4)=0$. Comparing (\ref{alpha defn}) with Definition \ref{Maslov defn num}, we see that \begin{equation}
\mu(\alpha_1)=\Maslov.
\end{equation} Taken together with (\ref{additivity equation}), these observations show that \begin{equation}
|\Maslov|=|\mu(\alpha_2)|.
\end{equation} To prove Theorem \ref{Morse-Maslov ineq}, it therefore suffices to show that \begin{equation}
|\mu(\alpha_2)|\leq\Mor(L).
\end{equation} 
\begin{center}
	\begin{figure}[h]
		\includegraphics[scale=1.25]{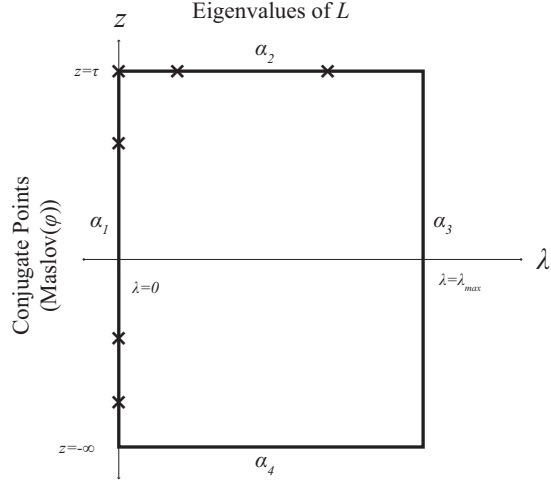}
		
		\caption{``Maslov Box": Domain in $\lambda z$-plane}
		
		\label{fig1}
	\end{figure}
\end{center} 
\begin{rem}
	Notice in Figure \ref{fig1} the conjugate point in the upper left corner. This crossing corresponds to the translation invariance of (\ref{tw pde}) (i.e. $E^u(0,\tau)\cap E^s(0,\tau)\neq\{0\}$). The contributions of this crossing to $\mu(\alpha_1)$ and $\mu(\alpha_2)$ can be determined using (\ref{rel Maslov index}).
\end{rem}Suppose that $\lambda^*$ is a conjugate point for $\alpha_2$. By definition, this means that \begin{equation}\label{alpha 2 crossing}
E^u(\lambda^*,\tau)\cap E^s(\lambda^*,\tau)\neq\{0\}.
\end{equation} But this is precisely the condition that $\lambda$ be an eigenvalue of $L$. Furthermore, the dimension of the intersection in (\ref{alpha 2 crossing}) captures the geometric multiplicity of $\lambda^*$ as an eigenvalue. By the triangle inequality, we therefore have \begin{equation}
|\mu(\alpha_2)|\leq\sum_{\lambda^*\in[0,\lambda_{\mathrm{max}}]}\dim(E^u(\lambda^*,\tau)\cap E^s(\lambda^*,\tau)),
\end{equation} where the sum is taken over all conjugate points. Since $[0,\lambda_\mathrm{max}]$ contains all possible real, unstable eigenvalues of $L$, and the geometric multiplicity of an eigenvalue is no greater than its algebraic multiplicity, we see that $|\mu(\alpha_2)|\leq\Mor(L)$, proving Theorem \ref{Morse-Maslov ineq}.

\section{Counting Eigenvalues in a FitzHugh-Nagumo System} 
There are two reasons that the inequality in Theorem \ref{Morse-Maslov ineq} cannot be improved to equality in general. First, the Maslov index counts \emph{signed} intersections, so that two different eigenvalues of $L$ might offset in the calculation of $\mu(\alpha_2)$ if the crossing forms have different signatures. Second, a given eigenvalue might be \emph{deficient} (i.e. have lesser geometric than algebraic multiplicity). In this and the next section, we consider traveling waves in a FitzHugh-Nagumo system wherein neither of these potential pitfalls occurs. Additionally, we prove that any unstable spectrum must be real, so that the Maslov index actually counts the total number of unstable eigenvalues.

The FitzHugh-Nagumo system is given by \begin{equation}\label{fhn pde}
\begin{aligned}
u_t & = u_{xx}+g(u)-v\\
v_t & = dv_{xx}+\eps(u-\gamma v),
\end{aligned}
\end{equation} where $g(u)=u(1-u)(u-a)$, $0<a<1/2$ and $\eps,\gamma>0$. Typically, $\eps$ is taken to be very small, and (\ref{fhn pde}) is studied using techniques of singular perturbation theory. The stability of various traveling and standing fronts and pulses has been studied for the variation of (\ref{fhn pde}) in which either $d=0$ or $0<d\ll 1$ \cite{Jones84,Yan85,Flo91,AGJ,Yana89}. If $d=1$, one checks that (\ref{fhn pde}) is of the form (\ref{gen pde}), with \begin{equation}
f(u)=\left(\begin{array}{c}
g(u)-v\\
-u+\gamma v
\end{array}\right), \hspace{.05 in} Q=\left(\begin{array}{c c }
1 & 0\\0 & -1
\end{array}\right), \hspace{.05 in} S=\left(\begin{array}{c c }
1 & 0 \\ 0 & \eps
\end{array}\right).
\end{equation}
The case $d=O(1)$ is considered in \cite{CH}, in which it is shown that standing waves for (\ref{fhn pde}) are stable. In \cite{CC15}, the authors use variational techniques to prove that traveling waves exist as well, but the stability question remains open. For $d=1$, the same traveling waves are constructed in \cite{CJ17b} using geometric singular perturbation theory. (See also \S 6 of \cite{CJ17}.) It falls out of this construction that the wave $\varphi=(\hat{u},\hat{v})$ moves to the left (i.e. $c<0$) and is homoclinic to $0$ as an orbit $(\hat{u},\hat{v},\hat{u}',\hat{v}'/\eps)$ in four-dimensional phase space. In an abuse of notation, we will use $\vp$ for both the solution $\vp=(\hat{u},\hat{v})$ of (\ref{fhn pde}) and the corresponding homoclinic orbit in phase space. In what follows, we show how the Maslov index provides the framework for showing that these waves are stable. The subsequent calculation of $\Maslov$, which completes the stability proof, is the topic of \cite{CJ17b}. To start, note that we are concerned with the spectrum of  \begin{equation}\label{fhn operator}
L_\eps=\p_z^2+c\p_z+\left(\begin{array}{c c}
g'(\hat{u}) & -1\\
\eps & -\eps\gamma
\end{array}\right),
\end{equation} acting on $BU(\bbR,\bbR^2)$. The subscript $\eps$ serves both to remind the reader that the operator is $\eps$-dependent and to distinguish results that are general for (\ref{gen pde}) from those that are specific to (\ref{fhn pde}). To apply the methods of this paper, we must verify that condition (\ref{Turing assumption}) is met. Indeed, a simple calculation gives that $g'(0)=-a$ and the eigenvalues of $QSf'(0)$ are \begin{equation}\label{fhn nonlinearity evals}
\nu_i=\frac{-(a+\eps\gamma)\pm\sqrt{(a+\eps\gamma)^2-4\eps(a\gamma+1)}}{2}.
\end{equation} For $\eps$ sufficiently small, the $\nu_i$ are easily seen to be negative and distinct. It then follows from (\ref{evals of A(lambda)}) that the eigenvalues of $A_\infty(\lambda)$ for $\lambda\geq 0$ are given by \begin{equation}\label{evals of A(lambda) fhn}
\mu_1(\lambda)<\mu_2(\lambda)<0<-c<\mu_3(\lambda)<\mu_4(\lambda).
\end{equation}
The benefit of having simple eigenvalues is that we can give analytically varying bases of $E^s(\lambda,z)$ and $E^u(\lambda,z)$ that separate solutions with different growth rates, see \cite{CJ17} for details. This will be important in \S 6 when we discuss the symplectic Evans function.

We now proceed to show that any unstable eigenvalues of $L_\eps$ must be real. After that, we address the issue of direction of crossings by deriving the $\lambda$ crossing form and showing that it is positive definite at all conjugate points. Finally, in \S 6 we show that the algebraic and geometric multiplicities of any unstable eigenvalues of $L_\eps$ are the same. This will prove:

\begin{theorem}\label{Morse = Maslov thm}
	\begin{equation}
	\Maslov=\Mor(L_\eps)=|\sigma(L_\eps)\cap\{\lambda\in\bbC:\mathrm{Re}\,\lambda\geq 0\}|.
	\end{equation}
\end{theorem}
\subsection{Realness of $\sigma(L_\eps)$}
For reference, we write out the eigenvalue problem for $L_\eps$ as a first order system, as in (\ref{eval eqn matrix}): \begin{equation}\label{eval eqn FHN}
\left(\begin{array}{c}
p\\q\\r\\s
\end{array} \right)_z=\left(\begin{array}{c c c c}
0 & 0 & 1 & 0\\
0 & 0 & 0 & \eps\\
\lambda-g'(\hat{u}) & 1 & -c & 0\\
-1 & \frac{\lambda}{\eps}+\gamma & 0 & -c
\end{array}\right)\left(\begin{array}{c}
p\\q\\r\\s
\end{array}\right).
\end{equation} As in \S 2, we abbreviate (\ref{eval eqn FHN}) as \begin{equation}\label{eval eqn FHN abb}
Y'(z)=A(\lambda,z)Y(z).
\end{equation} The discussion of $\sigma_\mathrm{ess}(L)$ from \S 2 applies to $L_\eps$ as well. However, the upper bound $\beta$ on the real part of the essential spectrum now depends on $\eps$. This is not a problem, since $\eps$ is fixed in the stability analysis. However, a few of the results to follow need $\eps$ to be ``sufficiently small." For completeness, we record the following lemma on $\sigma_\mathrm{ess}(L_\eps)$. \begin{lemma}
	For each $\eps>0$ sufficiently small, there exists $\beta_\eps<0$ such that \begin{equation}
	\sigma_\mathrm{ess}(L_\eps)\subset H_\eps:=\{\lambda\in\bbC:\mathrm{Re}\,\lambda<\beta_\eps\}.
	\end{equation}
\end{lemma}

 The analysis of $L_\eps$ is complicated by the presence of the $\p_z$ term in (\ref{fhn operator}). We can sidestep this difficulty by considering instead the operator \begin{equation}\label{Lc defn}
L_c:=e^{cz/2}L_\eps e^{-cz/2},
\end{equation} as is done in \cite{BJ,HLS16}. It is a routine calculation to see that for $(p, q)^T\in BU(\bbR,\bbC^2)$, we have \begin{equation}
L_c\left(\begin{array}{c}
p\\q
\end{array}\right)=\left(\begin{array}{c}
p_{zz}+\left(g'(\hat{u})-\frac{c^2}{4}\right)p-q\\
q_{zz}+\eps p-\left(\frac{c^2}{4}+\eps\gamma\right)q
\end{array}\right).
\end{equation}
Furthermore, if $L_\eps P=\lambda P$, then $L_c(e^{cz/2}P)=\lambda e^{cz/2}P$. This proves that the eigenvalues of $L_\eps$ and $L_c$ are the same, provided that $e^{cz/2}P$ is bounded for a given eigenvector $P$ of $L_\eps$. This is clearly the case as $z\rightarrow\infty$ since $c<0$. For the other tail, let $\lambda\in\bbC\setminus H_\eps$ be an eigenvalue of $L_\eps$ with associated eigenvector $P$. Since $A_\infty(\lambda)$ is hyperbolic with simple eigenvalues, $P$ must decay at least as fast as $e^{\mu_3(\lambda)z}$ as $z\rightarrow-\infty$. It follows that $e^{cz/2}P$ is bounded at $-\infty$ if \begin{equation}
\frac{c}{2}+\mu_3(\lambda)>0.
\end{equation} 
This is indeed the case, by (\ref{evals of A(lambda) fhn}). We therefore consider the eigenvalue problem \begin{equation}\label{eval problem 2}
L_cP=\lambda P.
\end{equation}
Making the change of variables $\tilde{q}=\frac{1}{\sqrt{\eps}}q$, we can rewrite (\ref{eval problem 2}) as (dropping the tildes) \begin{equation}\label{L_c eval prob}
\left(\begin{array}{c c}
\p_z^2+\left(g'(\hat{u})-\frac{c^2}{4}\right) & -\sqrt{\eps}\\
\sqrt{\eps} & \p_z^2-\left(\frac{c^2}{4}+\eps\gamma\right)
\end{array}\right)\left(\begin{array}{c}
p\\q
\end{array}\right)=\lambda\left(\begin{array}{c}
p\\q
\end{array}\right).
\end{equation}
$L_c$ is now seen to be of the form \begin{equation}
L_c=\left(\begin{array}{c c}
L_p & -\sqrt{\eps}\\
\sqrt{\eps} & L_q
\end{array}\right),
\end{equation} where $L_{p/q}$ are self-adjoint on $H^1(\bbR)$. Since any eigenfunction of $L_c$ in $BU(\bbR,\bbC^2)$ is exponentially decaying (provided $\lambda\in\bbC\setminus H_\eps$), we are free to consider the spectrum of $L_c$ as an operator on the Hilbert space $H^1(\bbR,\bbC^2)$ instead. The payoff of studying $L_c$ instead of $L_\eps$ is the following result, a version of which was proved in Lemma 4.1 of \cite{CH14}. We reproduce the proof here for convenience of the reader. \begin{lemma}\label{real spec lemma}
	For $\eps>0$ sufficiently small, if $\lambda\in\sigma_{n}(L_c)\cap(\bbC\setminus H_\eps)$ and $\mathrm{Re}\,\lambda\geq-\frac{c^2}{8} $, then $\lambda\in\bbR$. Consequently, the same is true for $L_\eps$.
\end{lemma}

\begin{proof}
	Let $\lambda=a+bi$ be an eigenvalue for $L_c$ with corresponding eigenvector $(p, q)^T$. Assume further that $a\geq -c^2/8$. Notice that the second equation in (\ref{L_c eval prob}) can be solved for $q$, since $L_q+c^2/8$ is negative definite (and hence $\lambda\notin\sigma(L_q))$. Explicitly, we have \begin{equation}
	q=-{\sqrt{\eps}}\left(L_q-a-bi\right)^{-1}p
	\end{equation}
	Next, substitute this expression into the first equation of (\ref{L_c eval prob}) to obtain \begin{equation}\label{specRe eqn 1}
	L_pp+\eps(L_q-a-bi)^{-1}p=(a+bi)p.
	\end{equation} Taking the $H^1$ pairing $\langle\cdot,\cdot\rangle$ with $p$ in (\ref{specRe eqn 1}) yields \begin{equation}\label{specRe eqn 2}
	\langle L_pp,p\rangle+\eps\langle (L_q-a-bi)^{-1}p,p\rangle=(a+bi)\langle p,p\rangle.
	\end{equation} Recalling that $L_p$ is self-adjoint, we extract the imaginary parts of (\ref{specRe eqn 2}): \begin{equation}\label{specRe eqn 3}
	\eps\, \mathrm{Im}\langle\left(L_q-a-bi\right)^{-1}p,p \rangle = b\langle p,p\rangle.
	\end{equation} The operator inverse in (\ref{specRe eqn 3}) can be decomposed into (self-adjoint) real and imaginary parts as follows: \begin{equation}\label{specRe eqn 4}
	(L_q-a-bi)^{-1}=\left((L_q-a)^2+b^2\right)^{-1}(L_q-a)+ib\left((L_q-a)^2+b^2\right)^{-1}.
	\end{equation} Combining (\ref{specRe eqn 3}) and (\ref{specRe eqn 4}), we arrive at \begin{equation}\label{specRe eqn 5}
	b\langle \left[\eps\left((L_q-a)^2+b^2\right)^{-1}-I\right]p,p\rangle=0,
	\end{equation} where $I$ denotes the identity operator. For operators on $H^1(\bbC,\bbC^2)$ we write $A<B$ if $(B-A)$ is positive definite. Since $L_q-a<0$ (independently of $\eps$) it follows from the inequality \begin{equation}
	\left((L_q-a)^2+b^2\right)^{-1}<(L_q-a)^{-2}
	\end{equation} and the fact that $(L_q-a)^{-2}$ is bounded that \begin{equation}
	\eps\left((L_q-a)^2+b^2\right)^{-1}-I<0
	\end{equation} for $\eps$ small enough. In conjunction with (\ref{specRe eqn 5}), this implies that $b=0$, as desired.
\end{proof}

\subsection{Monotonicity of $\lambda$-Crossings}
Recall that conjugate points along $\alpha_2$ correspond to eigenvalues of $L_\eps$. We will show below that the crossing form (in $\lambda$) is positive definite at all such crossings. This is the most significant difference between skew-gradient systems and the gradient systems considered in \cite{HLS16}, since the crossing form is always positive definite in the latter case (cf. \S 4.1 and \S 5.5). Conversely, we rely on the smallness of $\eps$ to get monotonicity of the crossings for $L_\eps$ in (\ref{fhn operator}). We stress that the $\lambda$ crossing form developed in this section would be the same for general systems (\ref{gen pde}). We focus on $L_\eps$ only because we are able to prove that the form is positive definite in this case.

To derive the $\lambda$ crossing form, we first take a closer look at the $z$ crossing form (\ref{crossing form z}). Suppose that $z^*$ is a conjugate point for $\alpha_1$, and that $\xi\in E^u(0,z^*)\cap E^s(0,\tau)$. By virtue of being in $E^u(0,z^*)$, we know that there exists a solution $u(z)$ of (\ref{eval eqn FHN}) such that $u(z)\in E^u(0,z)$ and $E^u(0,z^*)=\xi$. It follows that (\ref{crossing form z formula}) can be rewritten \begin{equation}
\Gamma(E^u(\lambda,\cdot),E^s(\lambda,\tau),z^*)(\xi)=\omega(\xi,A(0,z^*)\xi)=\omega(u(z),\p_z u(z))|_{z=z^*}.
\end{equation} In other words, the crossing form simplifies when evaluated on a vector that is part of a solution to a differential equation. Now suppose that $\lambda=\lambda^*$ is a conjugate point for $\alpha_2$, with $\xi\in E^u(\lambda^*,\tau)\cap E^s(\lambda^*,\tau)$. From (\ref{rel crossing form}), we know that we must evaluate two crossing forms--one where the curve $E^s(\lambda,\tau)$ is frozen at $\lambda=\lambda^*$ and one where $E^u(\lambda,\tau)$ is frozen at $\lambda=\lambda^*$. To simplify the calculations, we will work with $\Omega$ instead of $\omega$. Since one of these forms is just a scaled version of the other, it is clear that the signatures are the same, and hence the Maslov indices are as well. First consider the curve $\lambda\mapsto E^u(\lambda,\tau)$ and reference plane $E^s(\lambda^*,\tau)$. As in \S 3, we can write $E^u(\lambda,\tau)$ for $|\lambda-\lambda^*|$ small as the graph of an operator $B_\lambda:E^u(\lambda^*,\tau)\rightarrow J\cdot E^u(\lambda^*,\tau)$. This, in turn, generates a smooth curve $\gamma(\lambda)=(\xi+B_\lambda\xi)\in E^u(\lambda,\tau)$ with $\gamma(\lambda^*)=\xi$. By flowing backwards in $z$, we obtain a one-parameter family $u(\lambda,z)$ of solutions to (\ref{eval eqn FHN}) in $E^u(\lambda,\tau)$, with $u(\lambda^*,\tau)=\xi$. The same reasoning as for $z$ then shows that \begin{equation}\label{lambda xing form u}
\Gamma(E^u(\cdot,\tau),E^s(\lambda^*,\tau),\lambda^*)(\xi)=\Omega(u(\lambda,z),\p_\lambda u(\lambda,z))|_{\lambda=\lambda^*,z=\tau}.
\end{equation} The case where $E^s(\lambda,\tau)$ varies and $E^u(\lambda^*,\tau)$ is fixed is identical. We can generate a smooth family of solutions $v(\lambda,z)\in E^s(\lambda,z)$ with $v(\lambda^*,\tau)=\xi$. This half of the crossing form is then given by \begin{equation}\label{lambda xing form s}
\Gamma(E^s(\cdot,\tau),E^u(\lambda^*,\tau),\lambda^*)(\xi)=\Omega(v(\lambda,z),\p_\lambda v(\lambda,z))|_{\lambda=\lambda^*,z=\tau}.
\end{equation} By uniqueness of solutions, we importantly have \begin{equation}\label{eval condition}
u(\lambda^*,z)=v(\lambda^*,z):=P(z),
\end{equation} which is a $\lambda^*$-eigenvector of $L_\eps$. Putting together (\ref{rel crossing form}), (\ref{lambda xing form u}), and (\ref{lambda xing form s}), we see that \begin{equation}\label{lambda crossing form}
\begin{aligned}
\Gamma(E^u(\cdot,\tau),E^s(\cdot,\tau),\lambda^*)(\xi) & = \{\Omega(u(\lambda,z),\p_\lambda u(\lambda,z))-\Omega(v(\lambda,z),\p_\lambda v(\lambda,z))\}|_{\lambda=\lambda^*,z=\tau}\\
& = \{\Omega(u(\lambda,z),\p_\lambda u(\lambda,z))+\Omega(\p_\lambda v(\lambda,z),v(\lambda,z))\}|_{\lambda=\lambda^*,z=\tau}\\
& = \p_\lambda \Omega(v(\lambda,z),u(\lambda,z))|_{\lambda=\lambda^*,z=\tau},
\end{aligned}
\end{equation} where the last equality follows from (\ref{eval condition}).

The expression obtained in (\ref{lambda crossing form}) will be useful in the next section when we relate the crossing form to the Evans function. For now, we compute (\ref{lambda xing form u}) and (\ref{lambda xing form s}) directly. For (\ref{lambda xing form u}), we use the equality of mixed partials and the fact that $u$ solves (\ref{eval eqn FHN}) to obtain \begin{equation}\label{Lambda eqn}
(\p_\lambda u(\lambda,z))_z=A(\lambda,z)\p_\lambda u(\lambda,z)+A_\lambda u(\lambda,z),
\end{equation} where \begin{equation}
A_\lambda:=\p_\lambda A(\lambda,z)=\left(\begin{array}{c c c c}
0 & 0 & 0 & 0\\
0 & 0 & 0 & 0\\
1 & 0 & 0 & 0\\
0 & \eps^{-1} & 0 & 0
\end{array}\right).
\end{equation} Next, apply $\omega(u(\lambda,z),\cdot)$ to (\ref{Lambda eqn}) to see that \begin{equation}
\begin{aligned}
\omega(u(\lambda,z),A_\lambda u(\lambda,z)) & =\omega(u(\lambda,z),(\p_\lambda u(\lambda,z))_z)-\omega(u(\lambda,z),A(\lambda,z)\p_\lambda u(\lambda,z))\\
& = \omega(u(\lambda,z),(\p_\lambda u(\lambda,z))_z)+\omega(A(\lambda,z)u(\lambda,z),\p_\lambda u(\lambda,z))+c\omega(u(\lambda,z),\p_\lambda u(\lambda,z))\\
& =\p_z\omega(u(\lambda,z),\p_\lambda u(\lambda,z))+c\omega(u(\lambda,z),\p_\lambda u(\lambda,z)).
\end{aligned}
\end{equation} The second equality follows from the proof of Theorem \ref{L(n) invariant}, specifically (\ref{compatibility calc}).
Applying an integrating factor and using (\ref{Omega defn}) and (\ref{eval condition}) then shows that \begin{equation}\label{lambda crossing u}
\begin{aligned}
\Gamma(E^u(\cdot,\tau),E^s(\lambda^*,\tau),\lambda^*)(\xi) & =\Omega(u,\p_\lambda u)(\lambda^*,\tau)\\
& =\int\limits_{-\infty}^{\tau}\p_z\Omega(u(\lambda,z),\p_\lambda u(\lambda,z))|_{\lambda=\lambda^*}\,dz =\int\limits_{-\infty}^{\tau}e^{cz}\omega(P,A_\lambda P)\,dz.
\end{aligned}
\end{equation} The preceding calculation makes use of the fact that $u(\lambda,z)\in E^u(\lambda,z)$, and hence it decays faster than $e^{cz}$ as $z\rightarrow-\infty$, by (\ref{evals of A(lambda) fhn}). The calculation of the crossing form for the stable bundle using the solutions $v(\lambda,z)$ is identical until the last step. Indeed, those solutions decay at $+\infty$, so an application of the Fundamental Theorem gives \begin{equation}\label{lambda crossing s}
\begin{aligned}
\Gamma(E^s(\cdot,\tau),E^u(\lambda^*,\tau),\lambda^*)(\xi) & =\Omega(v,\p_\lambda v)(\lambda^*,\tau)\\
& =-\int\limits_{\tau}^{\infty}\p_z\Omega(v,\p_\lambda v)(\lambda^*,z)\,dz =-\int\limits_{\tau}^{\infty}e^{cz}\omega(P,A_\lambda P)\,dz.
\end{aligned}
\end{equation} Combining (\ref{rel crossing form}), (\ref{lambda crossing u}), and (\ref{lambda crossing s}), we see that the relative crossing form is given by \begin{equation}\label{lambda crossing form mono}
\Gamma(E^u(\cdot,\tau),E^s(\cdot,\tau),\lambda^*)(\xi)=\int_{-\infty}^{\infty}e^{cz}\,\omega(P,A_\lambda P)\,dz,
\end{equation} where $P\in E^u(\lambda^*,z)\cap E^s(\lambda^*,z)$ is the $\lambda^*$-eigenfunction of $L_\eps$ satisfying $P(\tau)=\xi$. Writing $P:=(p,q,p_z,q_z/\eps)$, it is straightforward to calculate from (\ref{symplectic form compatible}) that \begin{equation}\label{lambda crossing form integrand}
\omega(P,A_\lambda P)=p^2-\frac{q^2}{\eps}.
\end{equation} The following theorem shows that $\Gamma$ is positive definite for each conjugate point of $\alpha_2$, which proves that $\Maslov$ equals the sum of the geometric multiplicities of all unstable eigenvalues of $L_\eps$. 
\begin{theorem}\label{monotonicity thm}
	Let $\lambda\in\sigma(L_\eps)\cap(\bbR^{+}\cup\{0\})$ with corresponding eigenvector $P=(p,q)^T$. Suppose further that $0<\eps<\frac{c^4}{16}$. Then \begin{equation}
	\int\limits_{-\infty}^{\infty}e^{cz}\left(p^2-\frac{q^2}{\eps}\right)\,dz>0.
	\end{equation} In other words, the crossing form (\ref{lambda crossing form}) is positive definite for all $\lambda\in[0,\lambda_\mathrm{max}].$
\end{theorem} 
The proof of this theorem uses the following Poincar\'{e}-type inequality. 
\begin{lemma}\label{Poincare lma num}
	Suppose $h\in H^1(\bbR)$ satisfies \begin{equation}\label{weighted norm}
	\int\limits_{-\infty}^{\infty}e^{cz}\left(h^2+(h_z)^2\right)\,dz<\infty.
	\end{equation} Then for all $R\in\bbR$ (including $R=\infty$), we have \begin{equation}\label{Poincare lma}
	\frac{c^2}{4}\int\limits_{-\infty}^{R}e^{cz}h^2\,dz\leq\int\limits_{-\infty}^{R}e^{cz}(h_z)^2\,dz.
	\end{equation}
\end{lemma}
The proof of this inequality is a simple estimate using the fact that (\ref{weighted norm}) defines a norm on an exponentially weighted Sobolev space. For more details, we refer the reader to Lemma 4.1 of \cite{LMN04}, the source of this result. 

\begin{proof}[Proof of Theorem \ref{monotonicity thm}]

Written as a system, the eigenvalue equation $L_\eps P=\lambda P$ is \begin{equation}\label{eval problem written}
\begin{aligned}
p_{zz}+cp_z+(f'(\hat{u})-\lambda)p-q  =0\\
q_{zz}+cq_z+\eps p-(\eps\gamma+\lambda)q  =0.
\end{aligned}
\end{equation}
Now, multiply the second equation in (\ref{eval problem written}) by $e^{cz}q$ to obtain \begin{equation}\label{pos def calc 1}
\left(e^{cz}q_z\right)_zq-(\eps\gamma+\lambda)e^{cz}q^2=-e^{cz}\eps pq.
\end{equation} Since $p,q$ and their derivatives all decay exponentially in both tails, we can integrate (\ref{pos def calc 1}) to obtain (after an integration by parts) \begin{equation}\label{pos def calc 2}
\int\limits_{-\infty}^{\infty}e^{cz}(q_z)^2\,dz+(\eps\gamma+\lambda)\int\limits_{-\infty}^{\infty}e^{cz}q^2\,dz=\eps\int\limits_{-\infty}^{\infty}e^{cz}pq\,dz.
\end{equation} It then follows from (\ref{Poincare lma}), (\ref{pos def calc 2}), and the Cauchy-Schwarz inequality that \begin{equation}
\begin{aligned}
\frac{c^2}{4\eps}\int\limits_{-\infty}^{\infty}e^{cz}q^2\,dz & \leq\frac{1}{\eps}\int\limits_{-\infty}^{\infty}e^{cz}(q_z)^2\,dz\\
&<\int\limits_{-\infty}^{\infty}pq\,dz\leq\left(\int\limits_{-\infty}^{\infty}e^{cz}p^2\,dz\right)^{1/2}\left(\int\limits_{-\infty}^{\infty}e^{cz}q^2\,dz\right)^{1/2}.
\end{aligned}
\end{equation} Dividing the first and last terms in the inequality by $||q||_{1,c}$ (the $e^{cz}$-weighted $L^2$ norm) and squaring yields
\begin{equation}
\frac{c^4}{16\eps}\int\limits_{-\infty}^{\infty}e^{cz}\frac{q^2}{\eps}\,dz<\int\limits_{-\infty}^{\infty}e^{cz}p^2\,dz,
\end{equation} and the result now follows.
\end{proof}
\begin{rem}
The proof of the preceding theorem uses estimates that are very similar to calculations in \cite{CC15}. However, the objectives of the calculations are very different. In \cite{CC15}, the goal is to establish the existence of a traveling wave using variational techniques. Conversely, we are considering the stability issue, particularly what happens to eigenvalues as the spectral parameter varies.
\end{rem}
\section{Multiplicity of Eigenvalues: the Evans Function}

There is one remaining loose end to tie up if we want the Maslov index to give a complete picture of the unstable spectrum of $L_\eps$, namely, the multiplicity of eigenvalues. In general, for $\lambda\in\sigma_n(L)$, the \emph{geometric multiplicity} of $\lambda$ is given by $\dim\ker (L-\lambda I)$. For $\lambda\in \bbC\setminus H_\eps\cap\sigma_n(L_\eps)$, this number is bounded above by two (or $n$, in the general setting of (\ref{gen pde})), since $E^u(\lambda,z)$ and $E^s(\lambda,z)$ are only two-dimensional. Since $\dim\ker(L-\lambda I)=\dim(E^u(\lambda,z)\cap E^s(\lambda,z))$, it is clear from (\ref{alpha defn}) that the dimension of a crossing for $\alpha_2$ gives the geometric multiplicity of $\lambda$. 

The algebraic multiplicity of $\lambda$, on the other hand, is trickier. It is given by $\dim\ker(L-\lambda I)^\alpha$, where $\alpha$ is the ascent of $\lambda$, i.e. the smallest $\alpha$ for which $\dim\ker(L-\lambda I)^\alpha=\dim\ker(L-\lambda I)^{a+1}$. See \S 6.D of \cite{AGJ} for more details. There is nothing obvious about the Maslov index that addresses the algebraic multiplicity of an eigenvalue. In \cite{HLS16}, self-adjoint operators are studied, and this issue is moot, since the two multiplicities coincide. However, for our purposes it is not obvious that the two multiplicities are the same.

One tool that demonstrably gives information about the algebraic multiplicity of eigenvalues is the Evans function \cite{AGJ,Sandstede02}. Briefly, the Evans function $D(\lambda)$ is a Wronskian-type determinant that detects linear dependence between the sets $E^s(\lambda,z)$ and $E^u(\lambda,z)$. Thus it is zero if and only if $\lambda$ is an eigenvalue of $L$. It is also true (\S 2.E of \cite{AGJ}) that the order of $\lambda$ as a root of $D$ is equal to the algebraic multiplicity of $\lambda$ as an eigenvalue of $L$. This is the key to relating the Maslov index to algebraic multiplicity, as a symplectic version of the Evans function was developed in \cite{BD99,BD01}. The Maslov index in particular was used in Evans function analyses in \cite{CB14,CJ17}. 

The Evans function for (\ref{fhn operator}) is developed in detail in \cite{CJ17}, so we refer the reader there for more background. Since the eigenvalues of $A_\infty(\lambda)$ are real and simple for real $\lambda\geq 0$, we can find solutions $u_i(\lambda,z)$, $i=1\dots4$ to (\ref{eval eqn FHN}) such that \begin{equation}\label{individual soln decay}
\begin{aligned}
\lim\limits_{z\rightarrow\infty}e^{-\mu_i(\lambda)z}u_i(\lambda,z) & =\eta_i(\lambda), \hspace{.2 in}i=1,2\\
\lim\limits_{z\rightarrow-\infty}e^{-\mu_i(\lambda)z}u_i(\lambda,z) & = \eta_i(\lambda), \hspace{.2 in}i=3,4,
\end{aligned}
\end{equation} where $\eta_i(\lambda)$ is a nonzero eigenvector of $A_\infty(\lambda)$ corresponding to eigenvalue $\mu_i(\lambda)$. We can then write \begin{equation}
\begin{aligned}
E^s(\lambda,z)=\mathrm{sp}\{u_1(\lambda,z),u_2(\lambda,z) \}\\
E^u(\lambda,z)=\mathrm{sp}\{u_3(\lambda,z),u_4(\lambda,z) \}
\end{aligned}.
\end{equation} Although the Evans function can be defined without picking bases of the stable and unstable bundles (e.g. \cite{AGJ}), the symplectic structure cannot be exploited without isolating particular solutions. In \cite{CJ17}, these bases are used to define the Evans function as follows. \begin{define}
	The \textbf{Evans function} $D(\lambda)$ for (\ref{fhn operator}) is given by \begin{equation}\begin{aligned}
	\label{Evans fct defn}
	D(\lambda) & =e^{2cz}\det\left[u_1(\lambda,z),u_2(\lambda,z),u_3(\lambda,z),u_4(\lambda,z)\right] \\
	&= -\det\left[\begin{array}{c c}
	\Omega(u_1(\lambda,z),u_3(\lambda,z)) & \Omega(u_1(\lambda,z),u_4(\lambda,z))\\
	\Omega(u_2(\lambda,z),u_3(\lambda,z)) & \Omega(u_2(\lambda,z),u_4(\lambda,z))
	\end{array}\right]
	\end{aligned}.
	\end{equation} 
\end{define} We call the second formulation of $D$ in (\ref{Evans fct defn}) the ``symplectic Evans function." As mentioned above, $D(\lambda)=0$ if and only if $\lambda\in\sigma_n(L_\eps)$, and the order of $\lambda$ as a root of $D$ is equal to its algebraic multiplicity as an eigenvalue of $L_\eps$. $D$ is also independent of $z$, which follows from Theorem \ref{form invariance thm}. The following theorem is the main result of this section.

\begin{theorem}\label{multiplicity theorem}
Let $\lambda^*\in\sigma_n(L_\eps)\cap (\bbC\setminus H_\eps)$. Then the geometric and algebraic multiplicities of $\lambda^*$ are equal. This is equivalent to $\lambda=\lambda^*$ being a regular conjugate point of $\alpha_2$.
\end{theorem}
\begin{proof}
We prove this separately for $\lambda$ with geometric multiplicity one and two. Recall that a crossing is regular if the associated crossing form (\ref{lambda crossing form}) is nondegenerate. First, suppose that $\lambda^*$ is an eigenvalue of $L_\eps$ with geometric multiplicity one. The goal is to show that $D'(\lambda^*)\neq 0$. Let $P(z)$ be a corresponding eigenfunction. We can perform a change of basis near $\lambda=\lambda^*$ so that \begin{equation}
\begin{aligned}
E^s(\lambda,z)=\mathrm{sp}\{U(\lambda,z),a_s(\lambda,z) \}\\
E^u(\lambda,z)=\mathrm{sp}\{V(\lambda,z),a_u(\lambda,z)\}
\end{aligned},
\end{equation} with $U(\lambda^*,z)=V(\lambda^*,z)=P(z)$. Doing so changes $D(\lambda)$ by multiplication with a nonzero analytic function $C(\lambda)$ (\S 4.1 of \cite{Sandstede02}). Since $D(\lambda^*)=0$, we have \begin{equation}
\frac{d}{d\lambda}\left[D(\lambda)C(\lambda)\right]|_{\lambda=\lambda^*}=D'(\lambda^*)C(\lambda^*),
\end{equation} so making this change of basis does not affect whether or not the derivative of $D$ at $\lambda^*$ vanishes. It therefore suffices to consider $\tilde{D}'(\lambda^*)$, with \begin{equation}
\tilde{D}(\lambda)=-\det\left[\begin{array}{c c}
\Omega(U(\lambda,z),V(\lambda,z)) & \Omega(U(\lambda,z),a_u(\lambda,z))\\
\Omega(a_s(\lambda,z),V(\lambda,z)) & \Omega(a_s(\lambda,z),a_u(\lambda,z))
\end{array}\right].
\end{equation} (See Theorem 2 and Corollary 1 of \cite{CJ17} for more details on the derivation of this formula.) The desired derivative is computed using Jacobi's formula (\S 8.3 of \cite{MN88}), and an identical calculation is carried out in equation (4.5) of \cite{CJ17}. The result is that \begin{equation}
D'(\lambda^*)=\Omega(a_s(\lambda,z),a_u(\lambda,z))\,\p_\lambda\Omega(U(\lambda,z),V(\lambda,z))|_{\lambda=\lambda^*,z=\tau}.
\end{equation} Define $\xi=P(\tau)$. Comparing with (\ref{lambda crossing form}), we see that \begin{equation}
\p_\lambda \Omega(U(\lambda,z),V(\lambda,z))|_{\lambda=\lambda^*,z=\tau}=-\Gamma(E^u(\cdot,\tau),E^s(\cdot,\tau),\lambda^*)(\xi),
\end{equation} which is nonzero, since $\lambda^*$ is a regular crossing by Theorem \ref{monotonicity thm}. It would follow that $D'(\lambda^*)\neq 0$, and hence that $\lambda^*$ is a simple eigenvalue of $L_\eps$, if $\Omega(a_s(\lambda^*,z),a_u(\lambda^*,z))\neq 0.$ It turns out that this is equivalent to $\lambda^*$ having geometric multiplicity one. Indeed, if $\Omega(a_s(\lambda^*,z),a_u(\lambda^*,z))=0$, then $\mathrm{sp}\{a_s(\lambda^*,z),a_u(\lambda^*,z) \}$ is a Lagrangian plane. A simple dimension-counting argument (cf. page 85 of \cite{CB14}) then implies that \begin{equation}
E^s(\lambda^*,z)=\mathrm{sp}\{U,a_s\}=\mathrm{sp}\{V,a_u\}=E^u(\lambda^*,z).
\end{equation} We now turn to the case where $\lambda^*$ is a two-dimensional crossing, meaning that \begin{equation}
\dim(E^u(\lambda^*,z)\cap E^s(\lambda^*,z))=2.
\end{equation} By making another change of basis if necessary, we are free to assume that \begin{equation}\label{2d crossing}
\begin{aligned}
u_1(\lambda^*,z)=u_4(\lambda^*,z)\\
u_2(\lambda^*,z)=u_3(\lambda^*,z)
\end{aligned}.
\end{equation} We then set $\xi_1=u_1(\lambda^*,\tau)$ and $\xi_2=u_2(\lambda^*,\tau)$. Since the algebraic multiplicity of $\lambda^*$ is no less than its geometric multiplicity, we know \emph{a priori} that $D'(\lambda^*)=0$. This is easily verified by applying the product rule to (\ref{Evans fct defn}), which is the zero matrix for $\lambda=\lambda^*$. What we need to verify is that $D''(\lambda^*)\neq 0$, and that this is equivalent to the regularity of the crossing form. To see this, we use (\ref{Evans fct defn}) to write out \begin{equation}\label{Evans det}
D(\lambda)=\Omega(u_1(\lambda,z),u_4(\lambda,z))\Omega(u_2(\lambda,z),u_3(\lambda,z))-\Omega(u_1(\lambda,z),u_3(\lambda,z))\Omega(u_2(\lambda,z),u_4(\lambda,z)).
\end{equation} Evaluating at $\lambda=\lambda^*$, each of the four terms in (\ref{Evans det}) is zero, using (\ref{2d crossing}) and the fact that $E^{u/s}(\lambda,z)$ are Lagrangian planes. As mentioned above, we can see from (\ref{Evans det}) that $D'(\lambda^*)=0$, since the derivative produces a series of four terms, each of which is a product with a factor of zero. Computing $D''(\lambda^*)$ from the general Leibniz rule, we see that the only surviving terms are those for which each factor in (\ref{Evans det}) is differentiated once. Explicitly, we compute that \begin{equation}\label{D'' calc}
\begin{aligned}
D''(\lambda^*)  = & 2\left\{\p_\lambda\Omega(u_1(\lambda,z),u_4(\lambda,z))\p_\lambda\Omega(u_2(\lambda,z),u_3(\lambda,z))\right.\\
& \left.-\p_\lambda\Omega(u_1(\lambda,z),u_3(\lambda,z))\p_\lambda\Omega(u_2(\lambda,z),u_4(\lambda,z)) \right\}|_{\lambda=\lambda^*,z=\tau}\\
& = -2\det\left[\begin{array}{c c}
\p_\lambda\Omega(u_1(\lambda,z),u_3(\lambda,z)) & \p_\lambda\Omega(u_1(\lambda,z),u_4(\lambda,z))\\
\p_\lambda\Omega(u_2(\lambda,z),u_3(\lambda,z)) & \p_\lambda\Omega(u_2(\lambda,z),u_4(\lambda,z))
\end{array}\right]|_{\lambda=\lambda^*,z=\tau}\\
&= 2\det\left[\begin{array}{c c}
\p_\lambda\Omega(u_1(\lambda,z),u_4(\lambda,z)) & \p_\lambda\Omega(u_1(\lambda,z),u_3(\lambda,z))\\
\p_\lambda\Omega(u_2(\lambda,z),u_4(\lambda,z)) & \p_\lambda\Omega(u_2(\lambda,z),u_3(\lambda,z))
\end{array}\right]|_{\lambda=\lambda^*,z=\tau}.
\end{aligned}
\end{equation} We see from (\ref{2d crossing}) that the last matrix (obtained by switching columns and taking a transpose in the previous line) is exactly the matrix of the crossing form $\Gamma$ in (\ref{lambda crossing form}). To say that $\Gamma$ is nondegenerate means that the determinant in (\ref{D'' calc}) is nonzero, hence $D''(\lambda^*)\neq0$, as desired.
\end{proof}
\begin{rem}
	Although we phrased the preceding theorem for the operator $L_\eps$, it is clear that the proof generalizes to (\ref{gen pde}). At an $n$-dimensional crossing $\lambda^*$, the first $(n-1)$ derivatives of $D(\lambda)$ are forced to vanish. The $n^{\text{th}}$ derivative will then contain a factor corresponding to the $\lambda$-crossing form $\Gamma$. The number of zeros of $\Gamma$ in normal form (page 186 of \cite{vinberg}) counts the discrepancy between the algebraic and geometric multiplicities of $\lambda^*$ as an eigenvalue of $L$.
\end{rem}
Although algebraic versus geometric multiplicity seems like a picayune detail, it is actually critical in the case $\lambda=0$. Indeed, this eigenvalue is always present for traveling waves in autonomous equations. For such waves in semilinear parabolic systems, we have the following well-known result (e.g. \cite{AGJ}). \begin{theorem}\label{nonlinear stab theorem}
	Suppose that the operator $L$ in (\ref{L defn}) satisfies \begin{enumerate}
		\item There exists $\beta<0$ such that $\sigma(L)\setminus\{0\}\subset\{\lambda\in\bbC:\mathrm{Re }\lambda<\beta\}$.
		\item $0$ is a simple eigenvalue.
	\end{enumerate} Then $\hat{u}$ is stable in the sense of Definition \ref{stability defn}.
\end{theorem}
It is possible that $E^u(0,z)\cap E^s(0,z)=\mathrm{sp}\{\vp'(z) \}$ is one dimensional, but that $\lambda=0$ is still not a simple eigenvalue. In \cite{AJ94} (pp. 57-60), it is shown that $\lambda=0$ is simple if and only if the wave is \emph{transversely constructed}, in the following sense. With the equation $c'=0$ appended to (\ref{tw ode}), $W^{cu}(0)$ and $W^{cs}(0)$ are each $(n+1)$-dimensional. The wave $\vp$ is said to be transversely constructed if $W^{cu}(0)$ and $W^{cs}(0)$ intersect transversely in $\bbR^{2n+1}$, and their (necessarily one-dimensional) intersection is $\vp(z)$. Thus the geometric interpretation of simplicity is that two manifolds intersect transversely in augmented phase space. 

By contrast, the understanding of simplicity afforded by the symplectic structure requires no variation in $c$. Instead, we see that $\lambda=0$ (or any other eigenvalue) is simple if the curve $\lambda\mapsto E^u(\lambda,\tau)$ transversely intersects the train of $E^s(0,\tau)$ for all sufficiently large $\tau$. To see this, notice that Theorem \ref{multiplicity theorem} proves that the eigenvalue is simple if and only if the relative crossing form (\ref{lambda crossing form}) is regular. But if the integral (\ref{lambda crossing form mono}) is nonzero, then so will be the integral in (\ref{lambda xing form u}) for $\tau$ large enough. Alternatively, for an eigenvalue $\lambda^*$ with geometric multiplicity one, being simple is equivalent to the curves $\lambda\mapsto E^u(\lambda,\tau),E^s(\lambda,\tau)$ intersecting non-tangentially at $\lambda=\lambda^*$.

\bibliographystyle{amsplain}
\bibliography{FHN_Maslov2}

\end{document}